\documentclass[11pt]{amsart}

\makeindex

\usepackage{amsfonts,graphics,amsmath,amsthm,amsfonts,amscd,amssymb,amsmath,latexsym,multicol,euscript}
\usepackage{epsfig,url}
\usepackage{flafter}
\usepackage{hyperref}
\usepackage[all,cmtip,line]{xy}
\usepackage{array}
\usepackage[english]{babel}
\usepackage{overpic}
\usepackage{subfig}
\usepackage{multirow}

\usepackage{pgf,tikz}
\usepackage{mathrsfs}
\usetikzlibrary{arrows}

\usepackage{wrapfig}

\usepackage{enumitem}

\theoremstyle{definition}

\theoremstyle{theorem}
\newtheorem{theoremalpha}{Theorem}
\theoremstyle{corollary}

\newtheorem{theorem}{Theorem}[section]
\newtheorem{main theorem}{Main Theorem}
\newtheorem{proposition}[theorem]{Proposition}

\newtheorem{corollary}[theorem]{Corollary}

\newtheorem{lemma}[theorem]{Lemma}

\newtheorem{theorem*}{Theorem}
\newtheorem{corollary*}[theorem*]{Corollary}
\newtheorem{conjecture*}[theorem*]{Conjecture}

\theoremstyle{definition}

\newtheorem{definition}[theorem]{Definition}
\newtheorem{definition-lemma}[theorem]{Definition-Lemma}

\newtheorem{remark}[theorem]{Remark}

\newcommand\R{\mathbb{R}}
\newcommand\Q{\mathbb{Q}}
\newcommand\Z{\mathbb{Z}}

\newcommand\eps{\varepsilon}
\renewcommand\epsilon{\varepsilon}

\newcommand{\mc}{\mathcal}

\DeclareMathOperator{\codim}{codim}
\DeclareMathOperator{\N^1}{N^1}

\DeclareMathOperator{\ord}{ord}
\DeclareMathOperator{\mult}{mult}

\DeclareMathOperator{\Supp}{Supp}

\DeclareMathOperator{\Cent}{Cent}
\DeclareMathOperator{\vol}{vol}
\DeclareMathOperator{\val}{val}

\DeclareMathOperator{\SB}{SB}

\newcommand{\bm}{\mathbf B_-}  

\newcommand{\bp}{\mathbf B_+}  
\newcommand{\okbd}{\Delta}
\newcommand{\okval}{\Delta^{\val}}
\newcommand{\oklim}{\Delta^{\lim}}
\newcommand{\oknum}{\Delta^{\text{num}}}

\def\Cent{\operatorname{Cent}}

\newcommand{\kappanu}{\kappa_\nu}

\begin{document}

\title{Okounkov bodies associated to pseudoeffective divisors II}

\author{Sung Rak Choi}
\address{Department of Mathematics, Yonsei University, Seoul, Korea}
\email{sungrakc@yonsei.ac.kr}

\author{Jinhyung Park}
\address{School of Mathematics, Korea Institute for Advanced Study, Seoul, Korea}
\email{parkjh13@kias.re.kr}

\author{Joonyeong Won}
\address{Center for Geometry and Physics, Institute for Basic Science, Pohang, Korea}
\email{leonwon@ibs.re.kr}

\date{\today}
\thanks{S. Choi and J. Park were partially supported by NRF-2016R1C1B2011446. J. Won was partially supported by IBS-R003-D1, Institute for Basic Science in Korea. }

\begin{abstract}
We first prove some basic properties of Okounkov bodies, and give a characterization of Nakayama and positive volume subvarieties of a pseudoeffective divisor in terms of Okounkov bodies. Next, we show that each valuative and limiting Okounkov bodies of a pseudoeffective divisor which admits the birational good Zariski decomposition is a rational polytope with respect to some admissible flag. This is an extension of the result of Anderson-K\"{u}ronya-Lozovanu about the rational polyhedrality of Okounkov bodies of big divisors with finitely generated section rings.
\end{abstract}

\maketitle


\section{Introduction}

This paper is a continuation of our investigation on Okounkov bodies associated to pseudoeffective divisors (\cite{CHPW1}, \cite{CHPW2}, \cite{CPW}).
Let $X$ be a smooth projective variety of dimension $n$, and $D$ be a divisor on $X$. Fix an admissible flag $Y_\bullet$ on $X$, that is, a sequence of irreducible subvarieties
$$
Y_\bullet: X=Y_0\supseteq Y_1\supseteq \cdots\supseteq Y_{n-1}\supseteq Y_n=\{ x\}
$$
where each $Y_i$ is of codimension $i$ in $X$ and is smooth at $x$. The Okounkov body $\okbd_{Y_\bullet}(D)$ of a big divisor $D$ with respect to $Y_\bullet$ is a convex body in the Euclidean space $\R^n$ which carries rich information of $D$.
Okounkov first defined the Okounkov body associated to an ample divisor in \cite{O1}, \cite{O2}. After this pioneering work, Lazarsfeld-Musta\c{t}\u{a} \cite{lm-nobody} and Kaveh-Khovanskii \cite{KK} independently generalized Okounkov's work to big divisors (see \cite{B2} for a survey). We then further extended the study of Okounkov bodies to pseudoeffective divisors in \cite{CHPW1}. More precisely, we have introduced and studied two convex bodies, called the \emph{valuative Okounkov body} $\okval_{Y_\bullet}(D)$ and the \emph{limiting Okounkov body} $\oklim_{Y_\bullet}(D)$ associated to a pseudoeffective divisor $D$. See Sections \ref{okbdsubsec} and \ref{nakpvssec} for definitions and basics on Okounkov bodies.

In this paper, we first prove supplementary results to \cite{CHPW1}.
Main theorems of \cite{CHPW1} and the subsequent results in this paper depend on the following property of the Okounkov body. This theorem is a generalization of \cite[Theorem 4.26]{lm-nobody} and \cite[Theorem 3.4]{Jow}.

\begin{theoremalpha}[=Theorem \ref{newtheorem}]\label{newthrm}
Let $X$ be a smooth projective variety of dimension $n$, and $D$ be a big divisor on $X$. Fix an admissible flag $Y_\bullet$ such that $Y_{n-k} \not\subseteq \bp(D)$. Then we have
$$
\okbd_{Y_{n-    k\bullet}}(D) = \okbd_{Y_\bullet}(D) \cap (\{ 0\}^{n-k} \times \R_{\geq 0}^{k}).
$$
\end{theoremalpha}

In \cite{CHPW1}, we proved that the Okounkov bodies $\okval_{Y_\bullet}(D)$ and $\oklim_{Y_\bullet}(D)$ encode nice properties of the divisor $D$ if the given admissible flag $Y_\bullet$ contains a Nakayama subvariety of $D$ or a positive volume subvariety of $D$ (see Theorem \ref{chpwmain}).
We show the following characterization of those special subvarieties in terms of Okounkov bodies.

\begin{theoremalpha}[=Theorem \ref{geomcrit}]\label{critintro}
Let $X$ be a smooth projective variety of dimension $n$, and $D$ be an $\R$-divisor on $X$. Fix an admissible flag $Y_\bullet$ such that $Y_n$ is a general point in $X$. Then we have the following:
\begin{enumerate}[leftmargin=0cm,itemindent=.6cm]
\item[$(1)$] If $D$ is effective, then $Y_\bullet$ contains a Nakayama subvariety of $D$ if and only if $\okval_{Y_\bullet}(D) \subseteq \{0 \}^{n-\kappa(D)} \times \R^{\kappa(D)}$.
\item[$(2)$] If $D$ is pseudoeffective, then $Y_\bullet$ contains a positive volume subvariety of $D$ if and only if $\oklim_{Y_\bullet}(D) \subseteq \{0 \}^{n-\kappanu(D)} \times \R^{\kappanu(D)}$ and $\dim \oklim_{Y_\bullet}(D)=\kappanu(D)$.
\end{enumerate}
\end{theoremalpha}

One of the most important properties one can probably expect a convex set in $\R^n$ to satisfy is rational polyhedrality.
However, the geometric structure of Okounkov body is rather wild.
It can be non-polyhedral even if the variety $X$ is a Mori dream space and a divisor $D$ is ample (see \cite[Subsection 6.3]{lm-nobody}, \cite[Section 3]{KLM}). However, Anderson-K\"{u}ronya-Lozovanu proved that if a big divisor $D$ has a finitely generated section ring $R(X, D):=\bigoplus_{m \geq 0} H^0(X, mD)$, then there exists an admissible flag $Y_\bullet$ such that the Okounkov body $\okbd_{Y_\bullet}(D)$ is a rational polytope (\cite[Theorem 1]{AKL}).
We also refer to \cite[Theorems 1.1 and 4.17]{CPW} and \cite[Corollary 4.5]{S} for more related results.

Our next aim is to generalize \cite[Theorem 1]{AKL} to the valuative and limiting Okounkov bodies.
We recall that when a divisor $D$ is big, it has a finitely generated section ring if and only if it admits the birational good Zariski decomposition (see \cite[III.1.17.Remark]{nakayama}). However, for a pseudoeffective divisor $D$, such equivalence no longer holds in general; $D$ admits the birational good Zariski decomposition if and only if $D$ has a finitely generated section ring and is abundant (see Proposition \ref{zdabfg}).
For the rational polyhedrality of the Okounkov bodies of pseudoeffective divisors, we assume the existence of good Zariski decomposition on some birational model instead of the finite generation condition.
See Subsection \ref{zdsubsec} for our definition of (good) Zariski decomposition.

\begin{theoremalpha}[=Corollary \ref{ratpolval} and Theorem \ref{ratsimlim}]\label{main1}
Let $X$ be a smooth projective variety, and $D$ be a pseudoeffective $\Q$-divisor on $X$ which admits the good birational Zariski decomposition.
Then each  Okounkov bodies  $\okval_{Y_\bullet}(D)$ and $\oklim_{Y_\bullet}(D)$ is rational polyhedral with respect to some admissible flag $Y_\bullet$.
\end{theoremalpha}

We expect that the rational polyhedrality of Okounkov body holds in more general situations.
There are examples of divisors which do not admit birational good Zariski decompositions, but whose associated Okounkov bodies are rational polyhedral (see Remark \ref{ratrem}).

To prove Theorem \ref{main1} for the case of valuative Okounkov bodies, we use the same idea as \cite[Proposition 4]{AKL}.
Using only the finite generation of section ring, we show the rational polyhedrality of the valuative Okounkov body with respect to an admissible flag taken by the intersections of general members of the linear series (see Theorem \ref{ratsimval}).
For the case of limiting Okounkov bodies, under the given assumption, we prove the statement by reducing to the rationality problem of the limiting Okounkov body on some high model $f \colon Y\to X$ where the good Zariski decomposition of $f^*D$ exists (see Theorem \ref{ratsimlim}).

\medskip

The organization of the paper is as follows. In Section \ref{prelimsec}, we collect basic facts on various notions that are used in the proofs. Next, in Section \ref{okbdsubsec}, we recall basic properties of Okounkov bodies, and prove  Theorem \ref{newthrm}. Then we study some properties of Nakayama subvarieties and positive volume subvarieties to show Theorem \ref{critintro} in Section \ref{nakpvssec}. Section \ref{ratsec} is devoted to showing Theorem \ref{main1}.

\subsection*{Acknowledgment}
We would like to thank the referee for providing numerous  helpful suggestions and comments and for pointing out several gaps in earlier versions of this manuscript.

\section{Preliminaries}\label{prelimsec}

In this section, we collect relevant facts which will be used later.
Throughout the paper, $X$ is a smooth projective variety of dimension $n$, and we always work over an algebraically closed field of characteristic zero.

\subsection{Asymptotic invariants}
We review basic asymptotic invariants of divisors, namely, the asymptotic base loci and volume functions.
The \emph{stable base locus} of an $\R$-divisor $D$ is defined as $\SB(D):= \bigcap_{D \sim_{\R} D' \geq 0} \Supp(D')$.
The \emph{augmented base locus} of an $\R$-divisor $D$ is defined as $\bp(D):=\bigcap_A\text{SB}(D-A)$
where the intersection is taken over all ample divisors $A$.
The \emph{restricted base locus} of an $\R$-divisor $D$ is defined as $\bm(D):=\bigcup_{A}\SB(D+A)$
where the union is taken over all ample divisors $A$.
Note that $\bp(D)$ and $\bm(D)$ depend only on the numerical class of $D$.
For details, we refer to \cite{elmnp-asymptotic inv of base} and \cite{lehmann-red}.

Now, let $V$ be an irreducible subvariety of $X$ of dimension $v$.
The \emph{restricted volume} of a $\Z$-divisor $D$ along $V$ is defined as
$
\vol_{X|V}(D):=\limsup_{m \to \infty} \frac{h^0(X|V,mD)}{m^v/v!}
$
where $h^0(X|V,mD)$ is the dimension of the image of the natural restriction map $\varphi \colon H^0(S,\mc O_X(D))\to H^0(V,\mc O_V(D))$.
The restricted volume $\vol_{X|V}(D)$ depends only on the numerical class of $D$, and one can uniquely extend it to a continuous function
$$
\vol_{X|V} \colon \text{Big}^V (X) \to \R
$$
where $\text{Big}^V(X)$ is the set of all $\R$-divisor classes $\xi$ such that $V$ is not properly contained in any irreducible component of $\bp(\xi)$. When $V=X$, we simply let $\vol_X(D):=\vol_{X|X}(D)$, and we call it the \emph{volume} of an $\R$-divisor $D$.
For more details on volumes and restricted volumes, see \cite{pos} and \cite{elmnp-restricted vol and base loci}.
Now assume that $V\not\subseteq\bm(D)$ for an $\R$-divisor $D$.
The \emph{augmented restricted volume} of $D$ along $V$ is defined as
$\vol_{X|V}^+(D):=\lim_{\eps\to 0+} \vol_{X|V}(D+\eps A)$
where $A$ is an ample divisor on $X$. The definition is independent of the choice of $A$.
Note that $\vol_{X|V}^+(D)=\vol_{X|V}(D)$ for $D \in \text{Big}^V (X)$.
This also extends uniquely to a continuous function
$$
\vol_{X|V}^+ \colon \overline{\text{Eff}}^V(X) \to \R
$$
where $\overline{\text{Eff}}^V(X) := \text{Big}^V(X) \cup \{ \xi \in \overline{\text{Eff}}(X) \setminus \text{Big}(X) \mid V \not\subseteq \bm(\xi)   \}$.
For $D\in \overline{\text{Eff}}^V(X)$, we have $\vol_{X|V}(D) \leq \vol_{X|V}^+(D) \leq \vol_{V}(D|_V)$, and both inequalities can be strict in general.
See \cite{CHPW1} for more details on augmented restricted volumes.

\subsection{Iitaka dimension}

Let $D$ be an $\R$-divisor on $X$.
Let $\mathbb N(D)=\{m\in\Z_{>0}|\; |\lfloor mD\rfloor|\neq\emptyset\}$.
For $m\in\mathbb N(D)$, we consider the rational map $\phi_{mD} \colon X \dashrightarrow Z_m \subseteq \mathbb P^{\dim|\lfloor mD\rfloor|}$ defined by the linear system $|\lfloor mD\rfloor |$.
The \emph{Iitaka dimension} of $D$ is defined as
$$
\kappa(D):=\left\{
\begin{array}{ll}
\max\{\dim\text{Im}(\phi_{mD}) \mid \;m\in\mathbb N(D)\}& \text{if }\mathbb N(D)\neq\emptyset\\
-\infty&\text{if }\mathbb N(D)=\emptyset.
\end{array}
\right.
$$
We remark that  the Iitaka dimension $\kappa(D)$ is not really an invariant of the $\R$-linear equivalence class of $D$.
Nonetheless, it satisfies the property that $\kappa(D)=\kappa(D')$ for effective divisors $D,D'$ such that $D\sim_{\R}D'$.

For another important invariant, we fix a sufficiently ample $\Z$-divisor $A$ on $X$. The \emph{numerical Iitaka dimension} of $D$ is defined as
$$
\kappanu(D):= \max\left\{k \in \Z_{\geq0} \left|\; \limsup_{m \to \infty} \frac{h^0(X, \lfloor mD \rfloor + A)}{m^k}>0 \right.\right\}
$$
if $h^0(X, \lfloor mD \rfloor + A)\neq\emptyset$ for infinitely many $m>0$, and we let $\kappanu(D):=-\infty$ otherwise.
The numerical Iitaka dimension $\kappanu(D)$ depends only on the numerical class $[D]\in\N^1(X)_{\R}$.

\begin{definition}
An $\R$-divisor $D$ is said to be \emph{abundant} if $\kappa(D)=\kappanu(D)$.
\end{definition}

By definition, $\kappa(D) \leq \kappanu(D)$ holds and the inequality can be strict in general.
However, $\kappanu(D)=\dim X$ if and only if $\kappa(D)=\dim X$.
We refer to  \cite{E}, \cite{lehmann-nu}, \cite{nakayama} for more detailed properties of $\kappa$ and $\kappanu$.

Recall that the \emph{section ring of an $\R$-divisor $D$} is defined as $R(X, D):=\bigoplus_{m \geq 0} H^0(X, \lfloor mD \rfloor)$.

\begin{proposition}[{\cite[Corollary 1]{MR}}]\label{semiampleabundant}
A $\Q$-divisor $D$ on $X$ is semiample if and only if it is nef, abundant, and its section ring is finitely generated.
\end{proposition}

\subsection{Zariski decomposition}\label{zdsubsec}
We now briefly recall several notions related to Zariski decompositions in higher dimension.
For more details, we refer to \cite{B1}, \cite{nakayama}, \cite{P}.

To define the divisorial Zariski decomposition, we first consider a divisorial valuation $\sigma$ on $X$ with the center $V:=\Cent_X \sigma$ on $X$.
If $D$ is a big $\R$-divisor on $X$, we define \emph{the asymptotic valuation} of $\sigma$ at $D$ as
$\ord_V(||D||):=\inf\{\sigma(D')\mid D\equiv D'\geq 0\}$.
If $D$ is only a pseudoeffective $\R$-divisor on $X$, we define
$\ord_V(||D||):=\lim_{\epsilon\to 0+}\ord_V(||D+\eps A||)$
for some ample divisor $A$ on $X$. This definition is independent of the choice of $A$.
The \emph{divisorial Zariski decomposition} of  a pseudoeffective $\R$-divisor $D$ is the decomposition
$$
D=P_{\sigma}+N_{\sigma}
$$
into the \emph{negative part}  $N_{\sigma}:=\sum_{\codim E=1}  \ord_E(||D||)E$
where the summation is over the codimension one irreducible subvarieties $E$ of $X$ such that $ \ord_E(||D||)>0$ and the \emph{positive part} $P_{\sigma}:=D-N_{\sigma}$.


Let $D$ be an $\R$-divisor on $X$ which is effective up to $\sim_{\R}$.
The \emph{$s$-decomposition} of $D$ is the decomposition
$$
D=P_s+N_s
$$
into the \emph{negative part} $N_s:=\inf\{L \mid L \sim_{\R} D, L \geq 0\}$ and the \emph{positive part} $P_s:=D-N_s$.
The positive part $P_s$ is also characterized as the smallest divisor such that $P_s \leq D$ and $R(X, P_s) \simeq R(X, D)$ (\cite[Proposition 4.8]{P}).
Note that $P_s \leq P_\sigma$ and $P_s, P_\sigma$ do not coincide in general.

\begin{lemma}\label{abundantdiv=s}
Let $D$ be an abundant $\R$-divisor on $X$ with the divisorial Zariski decomposition $D=P_\sigma + N_\sigma$ and the $s$-decomposition $D=P_s+N_s$. Then $P_\sigma=P_s$.
\end{lemma}

\begin{proof}
Let  $\sigma$ be a divisorial valuation on $X$ with $V=\Cent_X \sigma$.
\cite[Proposition 6.4]{lehmann-red} implies that $\inf_{m \in \Z_{>0}, D' \in |\lfloor mD\rfloor|} \frac{1}{m}\sigma(D) = \ord_V(||D||)$ holds.
Since $\inf_{m \in \Z_{>0}, D' \in |\lfloor mD\rfloor|} \frac{1}{m}\sigma(D)=\sigma(N_s)$, we see that $D=P_s + N_s$ is the divisorial Zariski decomposition.
\end{proof}

The \emph{Fujita-Zariski decomposition} of a pseudoeffective $\R$-divisor $D$ is the decompositon
$$
D=P_f+N_f
$$
into the effective \emph{negative part} $N_f$ and the nef \emph{positive part} $P_f$ such that if $f \colon Y \to X$ is a birational morphism from a smooth projective variety and $f^*D=P'+N'$ with $P'$ nef and $N' \geq 0$, then $P' \leq f^*P$.
By definition, the divisorial Zariski decomposition and $s$-decomposition uniquely exist, and the Fujita-Zariski decomposition is also unique if it exists.
Recall that the Fujita-Zariski decomposition does not exist in general even if we take the pullback on a sufficiently high model $f \colon \widetilde{X} \to X$ (see \cite[Chapter IV]{nakayama}).

It is unclear in general whether the Fujita-Zariski decomposition is the divisorial Zariski decomposition (cf. \cite[III.1.17.Remark (2)]{nakayama}). However, this holds when the divisor is abundant and the positive part is semiample.

\begin{proposition}\label{goodzd}
Let $D$ be an abundant $\Q$-divisor on $X$ having a decomposition $D=P+N$ into a nef divisor $P$ and an effective divisor $N$. Then the following are equivalent:
\begin{enumerate}[leftmargin=0cm,itemindent=.6cm]
\item[$(1)$] It is the divisorial Zariski decomposition with $P=P_\sigma$ semiample.
\item[$(2)$] It is the Fujita-Zariski decomposition with $P=P_f$ semiample.
\item[$(3)$] It is the $s$-decomposition with $P=P_s$ semiample.
\end{enumerate}
\end{proposition}

\begin{proof}
$(1) \Rightarrow (2)$: It is easy to check that the divisorial Zariski decomposition with the nef positive part is the Fujita-Zariski decomposition (see \cite[III.1.17.Remark]{nakayama}).\\
$(2) \Rightarrow (3)$:
Let $D=P_s+N_s$ be the $s$-decomposition. Then $P_f \geq P_s$ by definition. Since $P_f$ is semiample, we also have $P_f \leq P_s$. Therefore $P_f=P_s$. \\
$(3) \Rightarrow (1)$: It follows from Lemma \ref{abundantdiv=s}.
\end{proof}


\begin{definition}\label{defbirgoodzd}
If one of the conditions in Proposition \ref{goodzd} holds for an abundant $\Q$-divisor $D$, then we say that $D$ \emph{admits the good Zariski decomposition}, and denote it by $D=P+N$.
We say that $D$ \emph{admits the birational good Zariski decomposition} if there exists a birational morphism $f \colon \widetilde{X} \to X$ from a smooth projective variety  such that $f^*D$ admits the good Zariski decomposition.
\end{definition}

\begin{proposition}\label{qgoodzd}
Let $D$ be a pseudoeffective $\Q$-divisor with the good Zariski decomposition $D=P+N$.
Then $P, N$ are also $\Q$-divisors.
\end{proposition}

\begin{proof}
Since $P$ is semiample, there exists a morphism $f \colon X \to Y$ such that $P \sim_{\R} f^* A$ where $A$ is an ample divisor on $Y$. The ample divisor $A$ can be written as a finite sum of ample Cartier divisors on $Y$ with positive real coefficients. Thus we can write $P \sim_{\R} \sum_{i=1}^k a_i P_i$ for some semiample Cartier divisors $P_i$ and some positive real numbers $a_i$.
Now we write $N = \sum_{j=1}^m b_j N_j$ for prime divisors $N_1, \ldots, N_m$ and positive real numbers $b_j$. Then $N_1, \ldots, N_m$ are  linearly independent in $\N^1(X)_{\R}$ by \cite[III.1.10.Proposition]{nakayama}.
Let $V_P$ and $V_N$ be the subspaces of $\N^1(X)_\R$ spanned by $\{P_i\}_{i=1}^k$ and $\{N_j\}_{j=1}^m$, respectively.
We now claim that $V_P\cap V_N=\{0\}$.
Suppose that the claim does not hold.
Then there exists a nonzero class $\eta\in V_P\cap V_N$ such that $\eta\equiv P'\equiv N'$
where $P'\in \bigoplus_{i=1}^k \R\cdot P_i$ and $N'\in \bigoplus_{j=1}^m \R\cdot N_j$.
Note that there exists a positive number $\epsilon>0$ such that for any real number $r$ satisfying $|r|<\epsilon$,
the divisor $P-r P'$ is nef and $N+r N'$ is effective.
Thus \cite[Proposition III.1.14 (2)]{nakayama} implies that in the following decompositions
$$
\begin{array}{rl}
D&=P+N\\
&\equiv (P-r P')+(N+r N'),
\end{array}
$$
we have $N\leq N+r N'$, hence $0\leq r N'$ for any $r$ such that $|r|<\epsilon$.
However, since $N'$ is a nonzero divisor, this is a contradiction.
The claim implies that if $D$ is a $\Q$-divisor, then so is $N$ in the decomposition $D = P+N$.
Therefore $P, N$ are both $\Q$-divisors.
\end{proof}

%
%
%
%

Now, we characterize when a divisor admits the birational good Zariski decomposition.

\begin{proposition}\label{zdabfg}
Let $D$ be a pseudoeffective $\Q$-divisor on $X$. Then
$D$ admits the birational good Zariski decomposition if and only if $D$ is abundant and $R(X, D)$ is finitely generated.
\end{proposition}

\begin{proof}
Suppose that there exists a birational morphism $f \colon \widetilde{X} \to X$ from a smooth projective variety such that $f^*D = P+N$ is the good Zariski decomposition. By definition, $D$ is abundant. Note that  $R(X, D) \simeq R(\widetilde{X}, f^*D) \simeq R(\widetilde{X}, P)$. Since $P$ is a semiample $\Q$-divisor by Proposition \ref{qgoodzd}, it follows from Proposition \ref{semiampleabundant} that $R(X, D)$ is finitely generated. Conversely, suppose that $D$ is abundant and $R(X, D)$ is finitely generated. For a sufficiently large and divisible integer $m>0$, we take a resolution $f \colon \widetilde{X} \to X$ of the base locus of $|mD|$ and consider the decomposition $f^*(mD)=M+F$ into the base point free $M$ and the fixed part $F$ of $|f^*mD|$
By the finite generation of $R(X, D)$, we see that $f^*D=\frac{1}{m}M+\frac{1}{m}F$ is the $s$-decomposition with semiample positive part. By Proposition \ref{goodzd}, $f^*D$ admits the good Zariski decomposition.
\end{proof}

\section{Okounkov bodies}\label{okbdsubsec}
In this section, we recall the construction of Okounkov bodies associated to pseudoeffective divisors in \cite{lm-nobody}, \cite{KK}, and \cite{CHPW1} and basic results. In the end, we prove Theorem \ref{newthrm} (=Theorem \ref{newtheorem}).


First, fix an admissible flag on $X$
$$
Y_\bullet: X=Y_0\supseteq Y_1\supseteq\cdots \supseteq Y_{n-1}\supseteq Y_n=\{x\}
$$
where each $Y_i$ is an irreducible subvariety of codimension $i$ in $X$ and is smooth at $x$.
Let $D$ be an $\R$-divisor on $X$ with $|D|_{\R}:=\{ D' \mid D \sim_{\R} D' \geq 0 \}\neq\emptyset$.
We define a valuation-like function
$$
\nu_{Y_\bullet}:|D|_{\R}\to \R_{\geq0}^n
$$
as follows.
For $D'\in |D|_\R$, let
$$\nu_1=\nu_1(D'):=\ord_{Y_1}(D').$$
Since $D'-\nu_1(D')Y_1$ is effective, we can define
$$\nu_2=\nu_2(D'):=\ord_{Y_2}((D'-\nu_1Y_1)|_{Y_1}).$$
If $\nu_i=\nu_i(D')$ is defined, then we define $\nu_{i+1}=\nu_{i+1}(D')$ inductively as
$$\nu_{i+1}(D'):=\ord_{Y_{i+1}}((\cdots((D'-\nu_1Y_1)|_{Y_1}-\nu_2Y_2)|_{Y_2}-\cdots-\nu_iY_i)|_{Y_{i}}).$$
The values $\nu_i(D')$ for $1 \leq i$ obtained as above define $\nu_{Y_\bullet}(D')=(\nu_1(D'),\nu_2(D'),\cdots,\nu_n(D'))$.

\begin{definition}
The \emph{Okounkov body} $\okbd_{Y_\bullet}(D)$ of a big $\R$-divisor $D$ with respect to an admissible flag $Y_\bullet$ is defined as the closure of the convex hull of $\nu_{Y_\bullet}(|D|_{\R})$ in $\R^n_{\geq 0}$.
\end{definition}

More generally, a similar construction can be applied to a graded linear series $W_\bullet$ on $X$ to construct the Okounkov body $\okbd_{Y_\bullet}(W_\bullet)$ of $W_\bullet$. For more details, we refer to \cite{lm-nobody}.

When $D$ is not big, we have the following extension introduced in \cite{CHPW1}.

\begin{definition}[{\cite[Definitions 1.1 and 1.2]{CHPW1}}]
Let $D$ be an $\R$-divisor on $X$.
\begin{enumerate}[leftmargin=0cm,itemindent=.6cm]
\item[(1)] When $D$ is effective up to $\sim_\R$, i.e., $|D|_{\R}\neq \emptyset$, the \emph{valuative Okounkov body} $\okval_{Y_\bullet}(D)$ of $D$ with respect to an admissible flag $Y_\bullet$ is defined as the closure of the convex hull of $\nu_{Y_\bullet}(|D|_{\R})$ in $\R^n_{\geq 0}$.
If $|D|_\R=\emptyset$, then we set $\okval_{Y_\bullet}(D):=\emptyset$.
\item[(2)] When $D$ is pseudoeffective, the \emph{limiting Okounkov body} $\oklim_{Y_\bullet}(D)$ of $D$ with respect to an admissible flag $Y_\bullet$ is defined as
$$\oklim_{Y_\bullet}(D):=\lim_{\epsilon \to 0+}\okbd_{Y_\bullet}(D+\epsilon A) = \bigcap_{\epsilon >0} \okbd_{Y_\bullet}(D+\epsilon A),$$
 where $A$ is an ample divisor on $X$.
(Note that $\oklim_{Y_\bullet}(D)$ is independent of the choice of $A$.)
If $D$ is not pseudoeffective, we set $\oklim_{Y_\bullet}(D)
:=\emptyset$.
\end{enumerate}\end{definition}

\begin{remark}
Boucksom's numerical Okounkov body $\oknum_{Y_\bullet}(D)$ in \cite{B2} is the same as our limiting Okounkov body $\oklim_{Y_\bullet}(D)$.
\end{remark}

Suppose that $D$ is effective.
By definition, $\okval_{Y_\bullet}(D) \subseteq \oklim_{Y_\bullet}(D)$, and the inclusion can be strict in general (see \cite[Examples 4.2 and 4.3]{CHPW1}).
Moreover, by \cite[Proposition 3.3 and Lemma 4.8]{B2}, we have
$$\dim \okval_{Y_\bullet}(D) = \kappa (D) \leq \dim \oklim_{Y_\bullet}(D) \leq \kappanu(D).$$

The following lemmas will be useful for computing Okounkov bodies.

\begin{lemma}\label{okbdbir}
Let $D$ be an $\R$-divisor on $X$. Consider a birational morphism $f : \widetilde{X} \to X$ with $\widetilde{X}$ smooth and an admissible flag
$$
\widetilde{Y}_\bullet : \widetilde{X}=\widetilde{Y}_0 \supseteq \widetilde{Y}_1 \supseteq \cdots \supseteq \widetilde{Y}_{n-1} \supseteq \widetilde{Y}_n=\{ x' \}.
$$
on $\widetilde{X}$. Suppose that $Y_n$ is a general point in $X$ and
$$
Y_\bullet:=f(\widetilde{Y}_\bullet) : X=Y_0 \supseteq Y_1=f(\widetilde{Y}_1) \supseteq \cdots \supseteq Y_{n-1}=f(\widetilde{Y}_{n-1}) \supseteq Y_n=f(\widetilde{Y}_n)=\{ f(x') \}.
$$ is an admissible flag on $X$. Then we have
$\okval_{\widetilde{Y}_\bullet}(f^*D)=\okval_{Y_\bullet}(D)$ and
$\oklim_{\widetilde{Y}_\bullet}(f^*D) = \oklim_{Y_\bullet}(D)$.
\end{lemma}

\begin{proof}
The limiting Okounkov body case is shown in \cite[Lemma 3.3]{CHPW2}. The proof for the valuative Okounkov body case is almost identical and we leave the details to the readers as an exercise.
\end{proof}

\begin{lemma}\label{okbdzd}
Let $D$ be an $\R$-divisor on $X$ with the $s$-decomposition $D=P_s+N_s$ and the divisorial Zariski decomposition $D=P_\sigma+N_\sigma$. Fix an admissible flag $Y_\bullet$ on $X$ such that $Y_n$ is a general point in $X$. Then we have $\okval_{Y_\bullet}(D)=\okval_{Y_\bullet}(P_s)$ and $\oklim_{Y_\bullet}(D)=\oklim_{Y_\bullet}(P_\sigma)$, respectively.
\end{lemma}

\begin{proof}
The first assertion follows from the fact that $R(X, D) \simeq R(X, P_s)$ and the construction of the valuative Okounkov body. The second assertion is nothing but \cite[Lemma 3.5]{CHPW2}.
\end{proof}

Finally, we give a proof of the main result of this section.
The following key result is implicitly used in \cite{CHPW1} (especially in the proof of \cite[Theorem B]{CHPW1}) and in this paper as well.
We include the complete proof here.

\begin{theorem}\label{newtheorem}
Let $X$ be a smooth projective variety of dimension $n$, and $D$ be a big divisor on $X$. Fix an admissible flag $Y_\bullet$ such that $Y_{n-k} \not\subseteq \bp(D)$. Then we have
$$
\okbd_{Y_{n-    k\bullet}}(D) = \okbd_{Y_\bullet}(D) \cap (\{ 0\}^{n-k} \times \R_{\geq 0}^{k}).
$$
\end{theorem}

\begin{proof}
We may assume that each $Y_i$ is a smooth variety.
Let $\{ A_i \}$ be a sequence of ample divisors on $X$ such that each $D+A_i$ is a $\Q$-divisor and $\lim\limits_{i \to \infty} A_i = 0$. Then we have
$$
\okbd_{Y_\bullet}(D) = \bigcap_{i=1}^{\infty} \okbd_{Y_\bullet}(D+A_i) \text{ and } \okbd_{Y_{n-k\bullet}}(D) = \bigcap_{i=1}^{\infty} \okbd_{Y_{n-k\bullet}}(D+A_i).
$$
Furthermore, $Y_{n-k} \not\subseteq \bp(D+A_i)$ for all $i$. Note that it is enough to prove the statement for the $\Q$-divisors $D+A_i$ for all sufficiently large $i$.
Thus we assume below that $D$ is a $\Q$-divisor.

It is easy to check that $\okbd_{Y_{n-k\bullet}}(D) \subseteq \okbd_{Y_\bullet}(D)$.
This implies that $\okbd_{Y_{n-k\bullet}}(D) \subseteq \okbd_{Y_\bullet}(D) \cap (\{ 0\}^{n-k} \times \R_{\geq 0}^{k})$  by definition.
Suppose that the inclusion is strict:
$$
\okbd_{Y_{n-k\bullet}}(D) \subsetneq \okbd_{Y_\bullet}(D) \cap (\{ 0\}^{n-k} \times \R_{\geq 0}^{k}).
$$
Then there exists a point $(0^{n-k}, x_1, \ldots, x_k) \in \okbd_{Y_\bullet}(D) \cap (\{ 0\}^{n-k} \times \R_{\geq 0}^{k})$, but $(0^{n-k}, x_1, \ldots, x_k) \not\in \okbd_{Y_{n-k\bullet}}(D)$.

Let $A$ be an ample $\Q$-divisor on $X$.
Note that $\okbd_{Y_{n-k\bullet}}(D) \subseteq \okbd_{Y_{n-k\bullet}}(D+\epsilon A)$ for any $\epsilon \geq 0$.
Since $Y_{n-k} \not\subseteq \bp(D+\epsilon A)$, we have $\vol_{\R^k} \okbd_{Y_{n-k\bullet}}(D+\epsilon A)= \frac{1}{(n-k)!}\vol_{X|Y_{n-k}}(D+\epsilon A)$. Recall that by \cite[Theorem A]{elmnp-restricted vol and base loci}, the function $\vol_{X|Y_{n-k}} \colon  \text{Big}^{Y_{n-k}}(X) \to \R$ is continuous,
where $\text{Big}^{Y_{n-k}}(X)$ denotes the cone in $\N^1(X)_\R$ consisting of the real divisor classes $\eta$ such that $Y_{n-k}$ is not properly contained in any of the irreducible components of $\bp(\eta)$.
Thus we can find a rational number $\epsilon >0$ such that $(x_1, \ldots, x_k)\not\in\okbd_{Y_{n-k\bullet}}(D+\epsilon A)$ and
$$
\vol_{\R^k} \okbd_{Y_{n-k\bullet}}(D+\epsilon A) < \vol_{\R^k} \Delta
$$
where $\Delta\subseteq \R^k$ is the convex hull of the set $\okbd_{Y_{n-k\bullet}}(D)$ and the point $(x_1, \ldots, x_k)$.
Note that we can fix a small neighborhood $U$ of $(x_1, \ldots, x_k)$ in $\R^k$ which is disjoint from $\okbd_{Y_{n-k\bullet}}(D+\epsilon A)$.

There exists a sufficiently small $\delta >0$ such that the divisors
$$
\begin{array}{rcl}
A_1=A_1(\delta_1)& \sim_{\Q} &\frac{1}{2}\epsilon A+\delta_1 Y_1,\\
A_2=A_2(\delta_1,\delta_2)& \sim_{\Q} &A_1|_{Y_1}+\delta_2 Y_2,\\
  &\vdots&\\
A_{n-k}=A_{n-k}(\delta_1,\delta_2,\ldots,\delta_{n-k})& \sim_{\Q} &A_{n-k-1}|_{Y_{n-k-1}}+\delta_{n-k} Y_{n-k}
\end{array}$$
are successively ample for any $\delta_j$ satisfying $\delta \geq \delta_1, \delta_2, \ldots, \delta_{n-k} >0$.
Since $(0^{n-k}, x_1, \ldots, x_k) \in \okbd_{Y_\bullet}(D)$, there exists a sequence of valuative points
$$\mathbf x_i=(\delta_1^i, \ldots, \delta_{n-k}^i, x_1^i, \ldots, x_k^i)\in\okbd_{Y_\bullet}(D)$$ such that
$$
\lim_{i \to \infty} \delta_j^i = 0 \text{ for $1 \leq j \leq n-k$\;\; and\; }\lim_{i \to \infty} x_l^i = x_l \text{ for $1 \leq l \leq k$}.
$$
Since it is known that the set of rational valuative points $\{\nu_{Y_\bullet}(D')| D\sim_\Q D'\geq 0\}$ is dense in $\okbd_{Y_\bullet}(D)$, we may assume that
$\mathbf x_i\in \{\nu_{Y_\bullet}(D')| D\sim_\Q D'\geq 0\}$ so that $\mathbf x_i\in\Q^n$ for all $i$.
We now fix a sufficiently large $i$ such that $0\leq \delta_j^i < \delta$ for all $1 \leq j \leq n-k$ and $(x_1^i, \ldots, x_k^i)$ lies in the small neighborhood $U$ in $\R^k$ of $(x_1, \ldots, x_k)$.
Since $\mathbf x_i$ is a rational valuative point of $\okbd_{Y_\bullet}(D)$, there exist an effective divisor $D'\sim_\Q D$ such that $\nu_{Y_\bullet}(D')=\mathbf x_i$. Namely, we have
$$
\begin{array}{rcl}
D'&=&D_1 + \delta_1^i Y_1,\\
D_1|_{Y_1}&=&D_2 + \delta_2^i Y_2,\\
&\vdots& \\
D_{n-k-1}|_{Y_{n-k-1}}&=&D_{n-k}+\delta_{n-k}^i Y_{n-k}
\end{array}
$$
where $D_j$ on $Y_{j-1}$ ($j=1,\ldots, n-k$) are effective divisors.

Now note that we have
$$
D'+\frac{1}{2}\epsilon A = D_1 + \left( \frac{1}{2}\epsilon A + \delta_1^i Y_1 \right) \sim_{\Q} D_1 + A'_1
$$
where we may assume that $A'_1$ is an effective ample divisor such that $\mult_{Y_1} A'_1=0$. We also have
$$
(D_1+A'_1)|_{Y_1} = D_2 + (A'_1|_{Y_1} + \delta_2^i Y_2) \sim_{\Q} D_2 + A'_2
$$
where we may assume that $A'_2$ is an effective ample divisor such that $\mult_{Y_2} A'_2=0$.
By continuing this process, we finally obtain
$$
(D_{n-k-1} + A'_{n-k-1})|_{Y_{n-k-1}} = D_{n-k} + (A'_{n-k-1}|_{Y_{n-k-1}} + \delta_{n-k}^i Y_{n-k}) \sim_{\Q} D_{n-k} + A'_{n-k}
$$
where we may assume that $A'_{n-k}$ is an effective ample divisor such that $\mult_{Y_{n-k}} A'_{n-k}=0$.


We now claim that there exists an effective divisor $D'' \sim_{\Q} D+\epsilon A$ such that $D''|_{Y_{n-k-1}} = D_{n-k}+E$ for some effective divisor $E$ with $\mult_{Y_{n-k}}E=0$ and $\nu_{Y_{n-k\bullet}}(E|_{Y_{n-k}}) = (x_1', \ldots, x_k')$ where we may assume that $x_j'\geq 0$ are arbitrarily small.
Note that  such $D''$ defines a rational valuative point $\nu_{Y_{\bullet}}(D'') = (0^{n-k}, x_1^i+x_1', \ldots, x_k^i + x_k') \in \okbd_{Y_\bullet}(D + \epsilon A)$. Thus $(x_1^i+x_1', \ldots, x_k^i + x_k') \in \okbd_{Y_{n-k \bullet}}(D+\epsilon A)$. If our claim holds, then we can conclude that $(x_1^i+x_1', \ldots, x_k^i + x_k')$ belongs to the small neighborhood $U$ of $(x_1, \ldots, x_k)$ in $\R^k$, which is a contradiction since $U$ is disjoint from $\okbd_{Y_{n-k\bullet}}(D+\epsilon A)$.
Therefore we finally obtain $\okbd_{Y_{n-k\bullet}}(D) = \okbd_{Y_\bullet}(D) \cap (\{ 0\}^{n-k} \times \R_{\geq 0}^{k})$.

It now remains to show the claim.
For a sufficiently divisible and large integer $m>0$, we take a log resolution $f_m \colon \widetilde{X}_m \to X$ of the base ideal of $|m(D+\frac{1}{2}\epsilon A)|$ so that we obtain a decomposition $f_m^*(m(D+\frac{1}{2}\epsilon A)) = M_m' + F_m'$ into a base point free divisor $M_m'$ and the fixed part $F_m'$ of $|f_m^*(m(D+\frac{1}{2}\epsilon A))|$. Let $M_m:=\frac{1}{m}M_m'$.
We may assume that $f_m$ is isomorphic outside $\bp(D+\frac{1}{2}\epsilon A)$.
We can take smooth strict transforms $\widetilde{Y}_i^m$ on $\widetilde{X}_m$ of $Y_i$ for $1 \leq i \leq n-k$.
For a general point $y$ in $\widetilde{Y}_{n-k}^m$, we have the positive moving Seshadri constant $\epsilon(||D+\frac{1}{2}\epsilon A||; f_m(y)) > 0$. Thus we also have the positive Seshadri constant $\epsilon(M_m; y) >0$ for $m \gg 0$ so that
 $\widetilde{Y}_{n-k}^m \not\subseteq \bp(M_m)$.
Let $g_m \colon \widetilde{X}_m \to Z_m$ be the birational morphism defined by $|M_m'|$.
Possibly by taking a further blow-up of $\widetilde{X}_m$, we may assume that every irreducible component of the exceptional locus of $g_m$ is a divisor.
We can still assume that $f_m$ is isomorphic over a general point in $Y_{n-k}$.
The divisor $H_m:=M_m - E_m$ is ample for any sufficiently small effective divisor $E_m$ whose support is the $g_m$-exceptional locus.
Note that $\mult_{\widetilde{Y}_{n-k}^m}(E_m)=0$.
Let $f_m^*(D+\frac{1}{2}\epsilon A)=P_m+N_m$ be the divisorial Zariski decomposition.
As in \cite[Proof of Proposition 3.7]{lehmann-nu}, by applying \cite[Proposition 2.5]{elmnp-asymptotic inv of base}, we see that $P_m - M_m$ is arbitrarily small if we take a sufficiently large $m>0$.
Since we may take an arbitrarily small $E_m$, so is $P_m - H_m$ for a sufficiently large $m>0$.

For simplicity, we fix a sufficiently large integer $m>0$ and we denote $f=f_m$, $\widetilde{X}=\widetilde{X}_m$ and $\widetilde{Y}_i=\widetilde{Y}_i^m$.
Let $f^*(D+\frac{1}{2}\epsilon A) = P+N$ be the divisorial Zariski decomposition. Then as we have seen above, we  can assume that $P$ can be arbitrarily approximated by an ample divisor $H$ on $\widetilde{X}$ such that $F=f^*(D+\frac{1}{2}\epsilon A)-H$ is an effective divisor satisfying  $\mult_{\widetilde{Y}_{n-k}}(F)=0$.
Note that $F-N$ is an arbitrarily small effective divisor such that $\mult_{\widetilde{Y}_{n-k}}(F-N)=0$.
Thus we can find an effective divisor $A_0 \sim_{\Q} A$ such that $\mult_{Y_{n-k-1}}A_0=0$, $E_0:=\frac{1}{2}\epsilon f^*A_0|_{\widetilde{Y}_{n-k-1}} - (F-N)|_{\widetilde{Y}_{n-k-1}}$ is effective, and $\mult_{\widetilde{Y}_{n-k}}E_0=0$.
Let $f^*D=P'+N'$ be the divisorial Zariski decomposition. Since $P'+f^*(\frac{1}{2}\epsilon A)$ is movable, we get $P \geq P'+f^*(\frac{1}{2}\epsilon A)$ and so $N' \geq N$.
Since $Y_{n-k} \not\subseteq \bp(D)$, every irreducible component of $N'$ cannot contain $\widetilde{Y}_{n-k-1}$.
Clearly, $f^*D_{n-k} - N'|_{\widetilde{Y}_{n-k-1}}$ is effective, and so is $f^*D_{n-k} - N|_{\widetilde{Y}_{n-k-1}}$.
Thus
$$
E_1:=f^*(D_{n-k} + A_{n-k}') -  N|_{\widetilde{Y}_{n-k-1}} + E_0 = f^*(D_{n-k}+A_{n-k}') - F|_{\widetilde{Y}_{n-k-1}} + \frac{1}{2}\epsilon f^*A_0|_{\widetilde{Y}_{n-k-1}}
$$
is an effective divisor on $\widetilde{Y}_{n-k-1}$.
Note that $E_1 \sim_{\Q} (H+\frac{1}{2}\epsilon f^*A)|_{\widetilde{Y}_{n-k-1}}$.
Since
$$
H^0\left(\widetilde{X}, m\left(H+\frac{1}{2}\epsilon f^*A\right)\right) \to H^0\left(\widetilde{Y}_{n-k-1}, m\left(H+\frac{1}{2}\epsilon f^*A\right)\Bigm|_{\widetilde{Y}_{n-k-1}}\right)
$$
 is surjective for all sufficiently divisible integers $m>0$, it follows that there exists $H' \sim_{\Q} H+\frac{1}{2}\epsilon f^*A$ such that $H'|_{\widetilde{Y}_{n-k-1}} = E_1$.
Then we have
$$
(H' + F)|_{\widetilde{Y}_{n-k-1}} = E_1 + F|_{\widetilde{Y}_{n-k-1}} = f^*D_{n-k} +E'
$$
where
$$
E':=f^*A_{n-k}' + (F-N)|_{\widetilde{Y}_{n-k-1}} + E_0 = f^*A_{n-k}' + \frac{1}{2}\epsilon f^*A_0|_{\widetilde{Y}_{n-k-1}}
$$
is an effective divisor.
Note that $\mult_{\widetilde{Y}_{n-k}}E'=0$.
We may also assume that each $x_j'\geq 0$ is arbitrarily small in $\nu_{\widetilde{Y}_{n-k\bullet}}(E'|_{\widetilde{Y}_{n-k}}) = (x_1', \ldots, x_k')$.
By letting $D'':=f_*(H'+F) \sim_{\Q} D + \epsilon A$ and $E:=f_*E'$, we obtain the divisors satisfying the required properties.
This shows the claim, and hence, we complete the proof.
\end{proof}

\section{Nakayama subvarieties and positive volume subvarieties}\label{nakpvssec}

In \cite{CHPW1}, we introduced Nakayama subvarieties and positive volume subvarieties of divisors. We now further study those subvarieties, and prove Theorem \ref{critintro}(=Theorem \ref{geomcrit}) in this section. We first recall the definitions of those subvarieties.

\begin{definition}[{\cite[Definitions 2.7 and 2.13]{CHPW1}}]
Let $D$ be an $\R$-divisor on $X$.
\begin{enumerate}[leftmargin=0cm,itemindent=.6cm]
\item[(1)] When $D$ is effective, a \emph{Nakayama subvariety of $D$} is an irreducible subvariety $U \subseteq X$ such that $\dim U=\kappa(D)$ and for every integer $m \geq 0$ the natural map
$$
H^0(X, \lfloor mD \rfloor) \to H^0(U, \lfloor mD|_U \rfloor)
$$
is injective (or equivalently, $H^0(X, \mc I_U \otimes \mc O_X(\lfloor mD \rfloor))=0$ where $\mc I_U$ is an ideal sheaf of $U$ in $X$).
\item[(2)] When $D$ is pseudoeffective, a \emph{positive volume subvariety of $D$} is an irreducible subvariety $V \subseteq X$ such that $\dim V = \kappanu(D)$ and $\vol_{X|V}^+(D)>0$.
\end{enumerate}
\end{definition}



\begin{remark}
In \cite{CHPW1}, we required an additional condition $V \not \subseteq \bm(D)$ for the definition of positive volume subvariety. However, we can drop this condition by Lemma \ref{notinbm}. Note that $V \not \subseteq \bm(D)$ does not imply $\vol_{X|V}^+(D)>0$ (see \cite[Example 2.14]{CHPW1}).
\end{remark}

\begin{lemma}\label{notinbm}
Let $D$ be a pseudoeffective $\R$-divisor on $X$. If $V$ is a positive volume subvariety of $D$, then $V \not\subseteq \bm(D)$.
\end{lemma}

\begin{proof}
If $V \subseteq \bm(D)$, then there is a sequence $\{ A_i \}$ of ample divisors on $X$ such that $\lim_{i \to \infty} A_i = 0$ and $V \subseteq \SB(D+A_i)$. Then $\vol_{X|V}(D+A_i)=0$, so $\vol_{X|V}^+(D)=0$. Thus $V$ is not a positive volume subvariety of $D$.
\end{proof}

\begin{remark}
Even if $V$ is a positive volume subvariety of $D$, it is possible that $V \subseteq \SB(D)$. For instance, consider a ruled surface $S$ carrying a nef divisor $D$ such that $D\cdot C>0$ for every irreducible curve $C \subseteq S$, but $D$ is not ample (see e.g., \cite[Example 1.5.2]{pos}). Since $\kappa(D)=-\infty$, we have $\SB(D)=S$. Thus every positive volume subvariety of $D$ is contained in $\SB(D)$. 
\end{remark}



\begin{remark}\label{gensub}
When $\kappa(D)=0$ (resp. $\kappanu(D)=0$), every point not in $\Supp(D)$ (resp. $\bm(D)$) is a Nakayama (resp. positive volume) subvariety of $D$.
When $\kappa(D)>0$, any  $\kappa(D)$-dimensional general subvariety (e.g., intersection of general ample divisors) is a Nakayama subvariety of $D$ (\cite[Proposition 2.9]{CHPW1}).
Similarly, when $\kappanu(D)>0$, any $\kappanu(D)$-dimensional intersection of sufficiently ample divisors is a positive volume subvariety of $D$ (\cite[Proposition 2.17]{CHPW1}).
In particular, we can always construct an admissible flag $Y_\bullet$ on $X$ containing a Nakayama subvariety of $D$ or a positive volume subvariety of $D$ such that $Y_n$ is a general point in $X$.
\end{remark}

The importance of such special subvarieties associated to divisors is that one can read off interesting asymptotic properties of divisors from Okounkov bodies with respect to admissible flags containing those subvarieties. The following theorem is the main result of \cite{CHPW1}, which can be regarded as a generalization of \cite[Theorem A]{lm-nobody}.

\begin{theorem}[{\cite[Theorems A and B]{CHPW1}}]\label{chpwmain}
We have the following:
\begin{enumerate}[leftmargin=0cm,itemindent=.6cm]
\item[$(1)$] Let $D$ be an effective $\R$-divisor on $X$. Fix an admissible flag $Y_\bullet$ containing a Nakayama subvariety $U$ of $D$ such that $Y_n$ is a general point in $X$. Then $\okval_{Y_\bullet}(D) \subseteq \{0 \}^{n-\kappa(D)} \times \R^{\kappa(D)}$ so that one can regard $\okval_{Y_\bullet}(D) \subseteq \R^{\kappa(D)}$. Furthermore, we have
$$\dim \okval_{Y_\bullet}(D)=\kappa(D) \text{ and } \vol_{\R^{\kappa(D)}}(\okval_{Y_\bullet}(D))=\frac{1}{\kappa(D)!} \vol_{X|U}(D).$$
\item[$(2)$] Let $D$ be a pseudoeffective $\R$-divisor on $X$, and fix an admissible flag $Y_\bullet$ containing a positive volume subvariety $V$ of $D$. Then $\oklim_{Y_\bullet}(D) \subseteq \{0 \}^{n-\kappanu(D)} \times \R^{\kappanu(D)}$ so that one can regard $\oklim_{Y_\bullet}(D) \subseteq \R^{\kappanu(D)}$. Furthermore, we have
$$\dim \oklim_{Y_\bullet}(D)=\kappanu(D) \text{ and } \vol_{\R^{\kappanu(D)}}(\oklim_{Y_\bullet}(D))=\frac{1}{\kappanu(D)!} \vol_{X|V}^+(D).$$
\end{enumerate}
\end{theorem}

\begin{remark}
To extract asymptotic properties of divisors from $\okval_{Y_\bullet}(D)$ as in Theorem \ref{chpwmain} (1), we need to assume that $Y_n$ is a general point in $X$.
When considering $\okval_{Y_\bullet}(D)$ (resp. $\oklim_{Y_\bullet}(D)$, we say that $Y_n$ is \emph{general} if $Y_n$ is not contained in $\SB(D)$ (resp. $\bm(D)$) (see \cite[Lemma 2.6]{lm-nobody} and \cite[Subsection 3.2]{CHPW1}).
\end{remark}

As an application of Theorem \ref{chpwmain}, we now prove the following Theorem \ref{critintro}.

\begin{theorem}\label{geomcrit}
Let $D$ be an $\R$-divisor on $X$. Fix an admissible flag $Y_\bullet$ such that $Y_n$ is a general point in $X$.
We have the following:
\begin{enumerate}[leftmargin=0cm,itemindent=.6cm]
\item[$(1)$]  If $D$ is effective, then $Y_\bullet$ contains a Nakayama subvariety of $D$ if and only if $\okval_{Y_\bullet}(D) \subseteq \{0 \}^{n-\kappa(D)} \times \R^{\kappa(D)}$.
\item[$(2)$] If $D$ is pseudoeffective, then $Y_\bullet$ contains a positive volume subvariety of $D$ if and only if $\oklim_{Y_\bullet}(D) \subseteq \{0 \}^{n-\kappanu(D)} \times \R^{\kappanu(D)}$ and $\dim \oklim_{Y_\bullet}(D)=\kappanu(D)$.
\end{enumerate}
\end{theorem}

\begin{proof}
The $(\Rightarrow)$ direction of both $(1)$ and $(2)$ at once follows from Theorem \ref{chpwmain}. For the $(\Leftarrow)$ direction of $(1)$, note that $\ord_{Y_{n-\kappa(D)}}(D')=0$ for every effective divisor $D' \sim_{\R} D$ under the assumption that
 $\okval_{Y_\bullet}(D) \subseteq \{0 \}^{n-\kappa(D)} \times \R^{\kappa(D)}$.
This means that $H^0(X, \mathcal{I}_{Y_{n-\kappa(D)}} \otimes \mathcal{O}_X(\lfloor mD \rfloor)) = 0$ for every integer $m \geq 0$. Thus $Y_{n-\kappa(D)}$ is a Nakayama subvariety of $D$.

For the $(\Leftarrow)$ direction of $(2)$, take an arbitrary ample divisor $A$ on $X$.
Since $\okbd_{Y_\bullet}(D+A) \supseteq \oklim_{Y_\bullet}(D)$, it follows that
$$
\okbd_{Y_\bullet}(D+A)\cap(\{0\}^{n-\kappanu(D)}\times\R_{\geq0}^{\kappanu(D)})\supseteq\oklim_{Y_{\bullet}}(D).
$$
Since $Y_n$ is general, we have $Y_{n-\kappanu(D)} \not\subseteq \bm(D)$. Thus $Y_{n-\kappanu(D)} \not\subseteq \bp(D+A)$ and using Theorem \ref{newtheorem}, we obtain $\okbd_{Y_{n-\kappanu(D)\bullet}}(D+A)\supseteq\oklim_{Y_{\bullet}}(D)$.
Therefore, by \cite[(2.7)]{lm-nobody} we have
$$
\begin{array}{rl}
\vol_{X|Y_{n-\kappanu(D)}}(D+A)&= \kappanu(D)!\cdot \vol_{\R^{\kappanu(D)}}\okbd_{Y_{n-\kappanu(D)\bullet}}(D+A)\\
&\geq  \kappanu(D)!\cdot \vol_{\R^{\kappanu(D)}}\oklim_{Y_{\bullet}}(D).
\end{array}
$$
The given condition implies that $\vol_{\R^{\kappanu(D)}}\oklim_{Y_{\bullet}}(D)>0$.
Hence, $\vol_{X|Y_{n-\kappanu(D)}}^+(D)>0$, and by definition $Y_{n-\kappanu(D)}$ is a positive volume subvariety of $D$.
\end{proof}

Regarding Theorem \ref{geomcrit} (1), we recall that $\dim \okval_{Y_\bullet}(D)=\kappa(D)$ always holds whenever $D$ is effective by \cite[Proposition 3.3]{B2}.




\section{Rational polyhedrality of Okounkov bodies}\label{ratsec}

This section is devoted to showing the rational polyhedrality of Okounkov bodies of pseudoeffective divisors. We then finally prove Theorem \ref{main1} (=Corollary \ref{ratpolval} and Theorem \ref{ratsimlim}). First, we  study the Okounkov bodies under surjective morphisms.

\begin{lemma}[{cf. \cite[Lemma 3.3]{CHPW2}}]\label{morokbd}
Let $f \colon X \to \overline{X}$ be a surjective morphism of projective varieties of the same dimension $n$, and fix an admissible flag
$$
Y_\bullet: X=Y_0\supseteq Y_1\supseteq\cdots \supseteq Y_{n-1}\supseteq Y_n=\{x\}
$$
on $X$ such that
$$
\overline{Y}_\bullet : \overline{X}=f(Y_0)\supseteq f(Y_1)\supseteq \cdots  \supseteq f(Y_{n-1}) \supseteq f(Y_n)=\{f(x) \}
$$
is an admissible flag on $\overline{X}$.
For a big $\Z$-divisor $D$ on $\overline{X}$, consider a graded linear series $W_\bullet$ associated to $f^*D$ on $X$ with $W_k:=H^0(\overline{X}, kD) \subseteq H^0(X, kf^*D)$ for any integer $k \geq 0$.
Then $\okbd_{Y_\bullet}(W_\bullet)=\okbd_{\overline{Y}_\bullet}(D)$.
\end{lemma}

\begin{proof}
It follows from the construction of Okounkov body associated to a graded linear series.
\end{proof}

The following lemma plays a crucial role in proving Theorem \ref{main1}.

\begin{lemma}[{cf. \cite[Proposition 4]{AKL}}]\label{simplex}
Let $W_\bullet$ be a graded linear series on a smooth projective variety $X$ generated by a base point free linear series $W_1$.
Suppose also that $W_1$ defines a surjective morphism $f \colon X \to \overline{X}$ of projective varieties of the same dimension $n$.
Let $Y_\bullet$ be an admissible flag on $X$ defined by successive intersection of sufficiently general members $E_1, \ldots, E_n$
of $W_1$ ; $Y_i := E_1 \cap \cdots \cap E_i$ for $1 \leq i \leq n-1$ and $Y_n = \{x \}$ is a general point in $X$. Then $\okbd_{Y_\bullet}(W_\bullet)$ is a $n$-dimensional simplex in $\R_{\geq 0}^n$ whose verticies are $0, e_1, \ldots, e_{n-1}, \vol_X(W_\bullet)e_n$.
\end{lemma}

\begin{proof}
There exists a very ample $\Z$-divisor $D$ on $\overline{X}$ so that we may assume $W_k=H^0(\overline{X}, kD) \subseteq H^0(X, kf^*D)$ for any integer $k \geq 0$. By the genericity assumption on $E_j$ for defining $Y_i$, we may assume that
$$
\overline{Y}_\bullet : \overline{X}=f(Y_0)\supseteq f(Y_1)\supseteq \cdots  \supseteq f(Y_{n-1}) \supseteq f(Y_n)
$$
is an admissible flag on $\overline{X}$. By Lemma \ref{morokbd}, $\okbd_{Y_\bullet}(W_\bullet)=\okbd_{\overline{Y}_\bullet}(D)$. Note that $D^n=\vol_{\overline{X}}(D)=\vol_X(W_\bullet)$. By applying \cite[Proposition 4]{AKL} to $\okbd_{\overline{Y}_\bullet}(D)$, we obtain the assertion.
\end{proof}

We now show the rational polyhedrality of $\okval_{Y_\bullet}(D)$.

\begin{theorem}\label{ratsimval}
Let $D$ be an effective $\Q$-divisor on $X$ with finitely generated section ring $R(X,D)$. Then there exists an admissible flag $Y_\bullet$ on $X$ containing a Nakayama subvariety of $D$ such that $\okval_{Y_\bullet}(D)$ is a rational simplex in $\{0\}^{n-\kappa(D)}\times\R^{\kappa(D)}$ of dimension $\kappa(D)$.
\end{theorem}

\begin{proof}
Let $m>0$ be a sufficiently divisible and large integer such that $mD$ is a $\Z$-divisor and the section ring $R(X, mD)$ is generated by $H^0(X, mD)$.
We take a log resolution $f \colon \widetilde{X} \to X$ of the base ideal $\frak{b}(|mD|)$ so that we obtain a decomposition $f^*(mD)=M+F$ into a base point free divisor $M$ and the fixed part $F$ of $|f^*(mD)|$. Note that the morphism $\phi \colon \tilde{X} \to Z$ given by $|M|$ is the Iitaka fibration of $f^*D$.
Let $A_1, \ldots, A_{n-\kappa(D)}$ be sufficiently general ample divisors on $\tilde{X}$ such that each $Y_i':=A_1 \cap \cdots \cap A_i$ for $1 \leq i \leq n-\kappa(D)$ is a smooth irreducible subvariety of dimension $n-i$. By Remark \ref{gensub}, $U:=Y_{n-\kappa(D)}'$ is a Nakayama subvariety of $f^*D$.
Let $W_k$ be the image of the natural injective map $H^0(\widetilde{X}, kf^*(mD)) \to H^0(U, kf^*(mD)|_U)$ for any integer $k \geq 0$. Then $W_\bullet$ is a graded linear series on $U$ generated by $W_1$.
Note that $\phi|_U \colon U \to Z$ is a surjective morphism of projective varieties of the same dimension $\kappa(D)$ defined by $W_1$.
Now take sufficiently general members $E_1, \ldots, E_{\kappa(D)}$ of $W_1$ such that $Y_{n-\kappa(D)+i}':=E_1 \cap \cdots \cap E_i$ for $1 \leq i \leq \kappa(D)-1$ is a smooth irreducible subvariety of $X$ (and $U$) of dimension $\kappa(D)-i$, and $Y_n'=\{x\}$ where $x$ is a general point in $U$. In particular, $Y_\bullet': Y'_0\supseteq \cdots \supseteq Y'_n$ is an admissible flag on $\widetilde{X}$ and the partial flag $Y'_{n-\kappa(D)\bullet}$ is an admissible flag on $U$.
Then by Lemma \ref{simplex}, $\okbd_{Y'_{n-\kappa(D)\bullet}}(W_\bullet)$ is a $\kappa(D)$-dimensional simplex.
Recall from \cite[Remark 3.11]{CHPW1} that $\okval_{Y'_\bullet}(f^*D)=\okbd_{Y'_{n-\kappa(D)\bullet}}(W_\bullet)$.
Furthermore, by the genericity assumption on $Y_\bullet'$, we can assume that $Y_\bullet : f(Y'_0)\supseteq \cdots\supseteq  f(Y'_n)$ is an admissible flag on $X$ and $f(Y_{n-\kappa(D)}')$ is a Nakayama subvariety of $D$.
By Lemma \ref{okbdbir}, $\okval_{Y_\bullet}(D)=\okval_{Y'_\bullet}(f^*D)$, and hence, $\okval_{Y_\bullet}(D)$ is a rational simplex. Finally, by Theorem \ref{chpwmain} (1), $\okval_{Y_\bullet}(D)$ is contained in $\{0\}^{n-\kappa(D)}\times\R^{\kappa(D)}$ and is of dimension $\kappa(D)$.
\end{proof}

\begin{corollary}\label{ratpolval}
Let  $D$ be an effective $\Q$-divisor on $X$ which admits the birational good Zariski decomposition. Then there exists an admissible flag $Y_\bullet$ on $X$ containing a Nakayama subvariety of $D$ such that $\okval_{Y_\bullet}(D)$ is a rational simplex in $\{0\}^{n-\kappa(D)}\times\R^{\kappa(D)}$ of dimension $\kappa(D)$.
\end{corollary}

\begin{proof}
By Proposition \ref{zdabfg}, $D$ has a finitely generated section ring. Then the assertion now follows from Theorem \ref{ratsimval}.
\end{proof}

We now turn to the limiting Okounkov body case.

\begin{lemma}\label{oklimnef}
Let $P$ be a nef divisor on $X$, and consider an admissible flag $Y_\bullet$ on $X$ containing a smooth positive volume subvariety $V=Y_{n-\kappanu(D)}$ of $P$. Then $\oklim_{Y_\bullet}(P)=\okbd_{Y_{n-\kappanu(P)\bullet}}(P|_V)$.
\end{lemma}

\begin{proof}
By definition, it is clear that $\oklim_{Y_\bullet}(P) \supseteq \okbd_{Y_{n-\kappanu(P)\bullet}}(P|_V)$. Thus it is sufficient to show that their Euclidean volumes in $\R^{\kappanu(P)}$ are equal,  i.e.,$ \vol_{\R^{\kappanu(P)}}(\oklim_{Y_\bullet}(P))=\vol_{\R^{\kappanu(P)}}(\okbd_{Y_{n-\kappanu(P)\bullet}}(P|_V))$, or equivalently, $\vol_{X|V}^+(P)=\vol_{V}(P|_V)$ by Theorem \ref{chpwmain}.
Fix an ample divisor $A$ on $X$.
Since $P+\eps A$ is ample for any $\eps >0$, it follows that $\vol_{X|V}(P+\eps A)=\vol_V((P+\eps A)|_V)$. By the continuity of the volume function, we obtain
$$
\vol_{X|V}^+(P)=\lim_{\eps \to 0+}\vol_{X|V}(P+\eps A)=\lim_{\eps \to 0+} \vol_V((P+\eps A)|_V)=\vol_V(P|_V),
$$
so we complete the proof.
\end{proof}

We next obtain an analogous result on the rational polyhedrality of $\oklim_{Y_\bullet}(D)$.

\begin{theorem}\label{ratsimlim}
Let $D$ be a pseudoeffective $\Q$-divisor on $X$ which admits the birational good Zariski decomposition. Then there exists an admissible flag $Y_\bullet$ on $X$ containing a positive volume subvariety of $D$ such that $\oklim_{Y_\bullet}(D)$ is a rational simplex in $\{0\}^{n-\kappanu(D)}\times\R^{\kappanu(D)}$ of dimension $\kappanu(D)$.
\end{theorem}

\begin{proof}
Let $f \colon \widetilde{X} \to X$ be a birational morphism of smooth projective varieties of dimension $n$ such that $f^*D=P+N$ is the good Zariski decomposition. Let $A_1, \ldots, A_{n-\kappanu(D)}$ be sufficiently general ample divisors on $\tilde{X}$ such that each $Y_i':=A_1 \cap \cdots \cap A_i$ for $1 \leq i \leq n-\kappanu(D)$ is a smooth irreducible subvariety of dimension $n-i$. By Remark \ref{gensub}, $V:=Y_{n-\kappanu(D)}'$ is a positive volume subvariety of $f^*D$.
By \cite[Theorem 2.18]{CHPW1}, $P|_V$ is big, and $mP|_V$ on $V$ is base point free for a sufficiently divisible and large integer $m>0$.
Let $E_1, \ldots, E_{\kappanu(D)-1} \in |mP|_V|$ be general members such that each $Y_{n-\kappanu(D)+i}':=E_1 \cap \cdots \cap E_i$ for $1 \leq i \leq \kappanu(D)-1$ is a smooth irreducible subvariety of $X$ of dimension $\kappanu(D)-i$, and $Y'_n:=\{ x\}$ where $x$ is a general point in $V$.
Then $Y'_\bullet : \widetilde{X}=Y'_0\supseteq \cdots\supseteq Y'_n$ is an admissible flag on $\widetilde{X}$. By \cite[Theorem 7]{AKL}, $\okbd_{Y'_{n-\kappanu(D)\bullet}}(P|_V)$ is a $\kappanu(D)$-dimensional simplex.
By Lemma \ref{oklimnef}, $\oklim_{Y'_\bullet}(P)=\okbd_{Y'_{n-\kappanu(D)\bullet}}(P|_V)$, and by Lemma \ref{okbdzd}, $\oklim_{Y'_\bullet}(f^*D)=\oklim_{Y'_\bullet}(P)$. By the genericity assumption on $Y'_\bullet$, we can assume that $Y_\bullet : f(Y'_0) \supseteq \cdots\supseteq f(Y'_n)$ is an admissible flag on $X$ and $f(Y_{n-\kappanu(D)}')$ is a positive volume subvariety of $D$.
By Lemma \ref{okbdbir}, we obtain $\oklim_{Y_\bullet}(D)=\oklim_{Y'_\bullet}(f^*D)$, and hence, $\oklim_{Y_\bullet}(D)$ is a rational simplex. Finally, by Theorem \ref{chpwmain}, $\oklim_{Y_\bullet}(D)$ is in $\{0\}^{n-\kappanu(D)}\times\R^{\kappanu(D)}$ and of dimension $\kappanu(D)$.
\end{proof}

\begin{remark}\label{ratrem}
The problem of the  rational polyhedrality of Okounkov body is not yet fully understood. It was shown in \cite[Corollary 13]{AKL} and \cite[Theorems 1.1 and 4.17]{CPW} that on a smooth projective surface, there always exists an admissible flag with respect to which the Okounkov body of any $\Q$-divisor is a rational polytope.
Thus, in particular, even if a pseudoeffective $\Q$-divisor is not abundant or does not have finitely generated section ring, the associated Okounkov body can still be a rational polytope with respect to some admissible flag.
On the other hand,  even when the given variety is  a Mori dream space, the Okounkov body can be non-polyhedral for some admissible flag (see \cite[Section 3]{KLM}).
\end{remark}

\end{document}